\documentclass[11pt]{article}
\usepackage{amsthm,amssymb,amsmath,lineno,cite}
\usepackage{epsfig}
\usepackage{bez123,calc,curves,ebezier,epic,eepic,graphicx,multiply,rotating}
\usepackage{color}
\usepackage{graphics}
\usepackage{makecell} 
\newtheorem{prelem}{{\bf Theorem}}

\newtheorem{theorem}{Theorem}

\newtheorem{lemma}[theorem]{Lemma}

\theoremstyle{definition}

\newtheorem{proposition}[theorem]{Proposition}

\theoremstyle{definition}
\theoremstyle{remark}

\newcommand{\diam}{{\rm diam}}
\newcommand{\tr}{{\rm Tr}}

\newcommand{\dd}{{\rm deg}}

\begin{document}
\title{{\bf Solving the Mostar index inverse problem}}
\date{}
\author{Yaser Alizadeh$^{1}$, Nino Ba\v{s}i\'c$^{2,3,4}$, Ivan Damnjanovi\'c$^{2,5,6}$, Tomislav Do\v{s}li\'c$^{7,8,}$\thanks{Corresponding author}, \\ Toma\v{z} Pisanski$^{2,4}$, Dragan Stevanovi\'c$^{9}\footnote{On leave from the Mathematical Institute of the Serbian Academy of Sciences and Arts}$, Kexiang Xu$^{10}$}
\maketitle
\begin{center}
    {\em $^1$Hakim Sabzevari University, Sabzevar, Iran}\\
    {\em $^2$University of Primorska, FAMNIT, Koper, Slovenia} \\
{\em $^3$IAM, University of Primorska, Koper, Slovenia} \\
{\em $^4$IMFM, Ljubljana, Slovenia}\\
    {\em $^5$University of Ni\v{s}, Faculty of Electronic Engineering, Ni\v{s}, Serbia}\\
    {\em $^6$Diffine LLC, San Diego, California, USA}\\
    {\em $^7$University of Zagreb Faculty of Civil Engineering, Zagreb, Croatia}\\
    {\em $^8$Faculty of Information Studies, Novo Mesto, Slovenia}\\
{\em $^9$Abdullah Al-Salem University, Khaldiya, Kuwait} \\
{\em $^{10}$School of Mathematics, Nanjing University of Aeronautics \& Astronautics, Nanjing, China}\\
\end{center}
\vskip 0.5cm
\begin{abstract}
\noindent 
A nonnegative integer $p$ is realizable by a graph-theoretical invariant $I$
if there exist a graph $G$ such that $I(G) = p$. The inverse problem for
$I$ consists of finding all nonnegative integers $p$ realizable by $I$.
In this paper, we consider and solve the inverse problem for the Mostar
index, a recently introduced graph-theoretical invariant which attracted a
lot of attention in recent years in both the mathematical and the chemical 
community. We show that a nonnegative integer is realizable by the 
Mostar index if and only if it is not equal to one. Besides presenting
the complete solution to the problem, we also present some empirical 
observations and outline several open problems and possible directions for
further research. \\ 

\end{abstract}
\newpage

\section{Introduction}

A {\em graph-theoretical invariant} is any number associated with a graph
which remains invariant under graph isomorphisms. In chemical graph
theory, such invariants are also known as {\em topological indices}.
Hundreds, if not thousands, of topological indices have been defined and
studied so far, mostly for their possible applications in the QSAR and
QSPR studies, but also for their intrinsic mathematical interest. This
paper is concerned with one such index which was recently introduced 
as a possible measure of edge peripherality and as a quantitative refinement
of distance-unbalancedness of graphs. The Mostar index, as it is called,
has attracted a lot of attention. For a thorough survey, we refer the
reader to \cite{alidoslic}, and for some more recent developments to
\cite{mikl,mikl1}. In this paper we study and solve the inverse
problem for the Mostar index.

For an integer-valued topological index $I$, the inverse problem for $I$ is to
determine whether, for a given integer $p$, there is a graph $G$ such that
$I(G) = p$. If such a graph $G$ exists, we say that $p$ is {\em realized}
by $I(G)$, or, more generally, that $p$ is {\em realizable} by $I$. Hence
the inverse problems are also known as the {\em realizability problems} for
topological indices.

Inverse problems were considered, and in
many cases solved, for most of the integer-valued topological indices.
See, for example, \cite{wang-2006} for the inverse problem for the Wiener
index of trees, \cite{yu-2019} for the Zagreb indices, and \cite{xu-2019}
for their differences. 

The paper is organized as follows. In the next section, we introduce 
necessary definitions, review some previous results on Mostar indices of
trees and refine them in a way useful for our main goals. In Section 3, 
we prove a sequence of results serving as the stepping stones to our
main result, the full solution of the inverse problem for the Mostar index.
In the last section we comment on some unsolved problems, offer some
partial results, and indicate several possible directions for future
research.

\section{Definitions and preliminary results}

All graphs considered in this paper are finite, simple and, unless stated
otherwise, also connected. For two vertices $u,v\in V(G)$ of a graph
$G=(V(G),E(G))$, $d_G(u,v)$ denotes the shortest-path distance between
$u$ and $v$. The {\em transmission} of a vertex $v\in V(G)$ is
$\tr _G(v)=\sum\limits_{u\in V(G)}d_G(u,v)$. Moreover, the {\em diameter}
of a graph $G$ is $\diam(G)=\max\limits_{\{u,v\}\subseteq V(G)}d_G(u,v)$.

The degree of a vertex $v$ is the number of edges incident with $v$; we denote
it by $\dd_G(v)$. Whenever $G$ is clear from the context, we simplify the
notation to $\dd(v)$. A vertex $v$ with $\deg_G(v)=1$ is called a
{\em pendent vertex} (also leaf if $G$ is a tree) in $G$. A connected graph
with maximum degree at most 4 is called a {\em chemical graph}. In particular,
chemical trees provide the graph representation of alkanes \cite{Gu-1986}.
A graph in which all vertices have the same degree $k$ is called a $k$-regular
graph. For other undefined notations and terminology on the graph theory,
please refer to \cite{bon-1976}.

Among the oldest and the best known topological indices is the Wiener index
\cite{wie-1947} $W(G)$ defined as
$W(G)=\frac{1}{2}\sum\limits_{v\in V(G)} \tr_G(v)$. Its introduction in
1947 went largely unnoticed, and the next two decades passed without much
activity. The situation changed radically in the early seventies, with the
introduction of the so-called Zagreb-group indices \cite{guttrina}, and
the Randi\'c index \cite{randic}, when their subsequent successful application
to the QSAR and QSPR studies, started a real avalanche. Many more indices
followed, and many of them proved useful in measuring and condensing the 
information coded in the connectivity patterns of molecular graphs.
This paper is concerned with one such index, the 
Mostar index, introduced recently in \cite{Tom2018} and defined as 
$$Mo(G)=\sum\limits_{uv\in E(G)}|n_u-n_v|,$$
where $n_u$ is the number of vertices in $G$ closer to $u$ than to $v$
and $n_v$ is defined analogously. The {\em contribution} of an edge
$e=uv$ to $Mo(G)$ of a graph is denoted by $\varphi_G(e)$. Hence
 $$Mo(G)=\sum\limits_{e\in E(G)}\varphi_G(e)$$
for any connected graph $G$. An edge $e=uv\in E(G)$ is {\em equieffective}
if $n_u=n_v$, that is, $\varphi_G(e)=0$, in $G$.  A graph in which each edge
is equieffective is called a {\em distance-balanced graph} \cite{je-2008}.
Examples are cycles $C_n$, complete graphs $K_n$, complete bipartite graphs
$K_{n,n}$ and cocktail party graphs $CP(n)$ with even $n$ (obtained by
removing a perfect matching from $K_n$). Some recent results on the Mostar
index can be found in \cite{gxd,te-2019,tra-2019}. 

It is important to notice that the restriction to simple graphs is essential.
For non-simple graphs, any integer $p > 1$ could be realized by a graph on
three vertices, two of them connected by $p-1$ parallel edges, and the third 
one attached to one of them by a pendent edge. Clearly, in this graph all
$p$ edges have their contributions equal to 1, and any integer greater than
one is realizable.

We end this section by stating some results on the Mostar indices of
trees. As usual, $P_n$ is the path on $n$ vertices, and $S_n$ denotes the
star on $n$ vertices, $S_n = K_{1,n-1}$.
A tree is {\em starlike} if it contains exactly one vertex of degree 
greater than 2. We will use the notation $T_n(k_1,\ldots,k_t)$ to denote
the starlike tree of order $n$ obtained by attaching
$t\ge 3$ paths of lengths $k_1,\ldots,k_t$, respectively, to a single
central vertex. 

The extremal values of the Mostar index over all trees on a given number 
of vertices are well known \cite{Tom2018}.
\begin{lemma}\label{ex-tree}
Let $T$ be a tree of order $n>3$. Then
$$\Big\lfloor\frac{(n-1)^2}{2}\Big\rfloor\le Mo(T)\le (n-1)(n-2),$$ with
the left equality if and only if $T\cong P_n$ and the right equality if
and only if  $T\cong S_n$.\end{lemma}

It is easy to check that both extremal values are even for all $n$. The
following result shows that this is also the case for all intermediate
values.
\begin{lemma}\label{even}
Let $T$ be a tree. Then $Mo(T)$ is an even number.
\end{lemma}
\begin{proof}
Let $n\ge 2$ be the order of $T$ with $e=uv\in E(T)$. If $n_u=m$, then
$|n_u - n_v| =|n-2m|$. It means that the contributions of all edges to
$Mo(T)$ have the same parity. Now, if $n$ is even (odd, respectively),
then $Mo(T)$ is the sum of $n-1$ even (odd, respectively) numbers.
This implies that $Mo(T)$ is an even number.
\end{proof}
Our following result refines Lemma \ref{ex-tree} by characterizing the
trees attaining the values near the lower end of the range.
\begin{proposition}\label{small}
Among all trees of order $n\geq 4$, the $(k+1)$-th smallest Mostar index
is attained for $T_n(1,k,n-2-k)$ with $1\le k\le \lfloor\frac{n-2}{2}\rfloor$.
\end{proposition}
\begin{proof}
From the structure of $T_n(1,k,n-2-k)$, we can assume that $T_n(1,k,n-2-k)$
consists of the longest path
$P_{n-1}=v_1v_2\cdots v_kv_{k+1}v_{k+2}\cdots v_{n-1}$ and a single vertex
$w$ with $wv_{k+1}\in E(T_n(1,k,n-2-k))$.  For simplicity, let
$T^*=T_n(1,k,n-2-k)$. Note that $1\le k\le \lfloor\frac{n-2}{2}\rfloor$
and $Mo(P_n)=\Big\lfloor\frac{(n-1)^2}{2}\Big\rfloor$. Then we have
\begin{eqnarray*}Mo(T_n(1,k,n-2-k))&=&\sum\limits_{v_{i}v_{i+1}\in E(T^*)}\phi_{T^*}(v_iv_{i+1})+\phi_{T^*}(v_{k+1}w)\\
&=&\sum\limits_{v_{i}v_{i+1}\in E(T^*)}\phi_{T^*}(v_iv_{i+1})+(n-2-2k)+2k\\
&=&\Big\lfloor\frac{(n-1)^2}{2}\Big\rfloor+2k.\end{eqnarray*}
By Lemmas \ref{even} and \ref{ex-tree}, we have the result as desired. 
\end{proof}
We will use the above result in our last section.

\section{Main result}

In this section we state and prove our main result. 
\begin{theorem}
A nonnegative integer $n$ is realizable by the Mostar index if and only if
$n \neq 1$.
\end{theorem}
We prove it in several
steps. First, we prove that all nonnegative integers different from 1, 3, 
or 5 are realizable, i.e., that they appear as Mostar indices of some graphs

\begin{lemma} \label{sviosim}
For every nonnegative integer $p \not\in \{1,3,5\}$ there is a simple
connected graph $G$ such that $Mo(G) = p$.
\end{lemma}
\begin{proof}
For $p = 0$ any cycle $C_n$ with $n \geq 3$ will do, as well as any complete
graph $K_n$, $n \geq 1$. $K_1$ is the smallest graph achieving this value,
while there is no largest graph with $Mo(G) = 0$.

For $p = 2$ take $G = P_3$.

If $p$ is an even number $p = 2m \geq 4$, take a complete graph $K_{m+1}$ and
attach a vertex $v$ to one of its vertices, say to the vertex $u$. The new
graph $G$ has $m + 2$ vertices. All edges of $K_{m+1}$ not incident with $u$
contribute $0$ to $Mo(G)$, while each of $m$ edges of $K_{m+1}$ incident
with $u$ contributes $1$. Finally, the contribution of the edge $uv$ is equal
to $m$. By summing all contributions we obtain
$$Mo(G) = m\cdot 1 + 1 \cdot m = 2m = p.$$

Let us look now at the odd numbers of the form $p = 4m - 1$ for $p \geq 7$,
i.e., for $m \geq 2$. Take an even cycle $C_{2m}$ and construct a new graph $G$
by attaching a new vertex $v$ to one of its vertices, say to the vertex $u$.
Again, the edge $uv$ contributes $2m - 1$ to $Mo(G)$, while each of the edges
of $C_{2m}$ contributes $1$. Hence,
$$Mo(G) = 2m \cdot 1 + 2m - 1 = 4m-1.$$
Finally, take an odd number $p = 4m+1$ for some $m \geq 2$. Take an odd
cycle $C_{2m+1}$ and label its vertices by $v_1, v_2, \ldots , v_{2m+1}$.
Attach a new vertex $v$ to $v_{m+1}$ and add an edge between $v_2$ and
$v_{2m+1}$. It can be verified by direct computation that only five edges
have nonzero contributions: The pending edge $v v_{m+1}$ contributes $2m$,
the new edge $v_2v_{2m+1}$ contributes $1$, as well as the edge
$v_{m+1}v_{m+2}$. The edge $v_1v_2$ contributes $m$, while the edge
$v_1v_{2m+1}$ contributes $m-1$. All other edges contribute zero. By adding
all contributions we obtain $Mo(G) = 4m+1 = p$. 
\end{proof}

The above constructions are not unique: The even numbers can be also realized
with fewer edges, starting from odd cycles and attaching one pendent vertex
(for $p = 8m$ and $p = 8m + 4$), two pendant vertices attached to the same
vertex of the cycle (for $p = 8m + 2$), and the numbers of the form $p = 8m + 6$
can be realized by taking an even cycle and identifying one of its vertices
with a vertex of a triangle. Those alternative constructions also show
that any integer different from 1, 3, and 5 can be realized by a chemical 
graph.

For the remaining three values, 1, 3, and 5, we first show that any graph
achieving them must be at least 2-connected and 2-edge-connected.

\begin{lemma} \label{2con}
Let $G$ be a simple connected  graph. If $Mo(G) \in \{1,3,5\}$, then $G$
has no bridge and no cut vertex. 
\end{lemma}
\begin{proof}
Let $n$ denote the order, i.e., the number of vertices of $G$, and let
$Mo (G) \in \{1,3,5\}$.

We start by checking all graphs of order at most $6$, and by observing that
none of them has the Mostar index $1,3$ or $5$. Moreover, we check all
graphs of order $7$ which have a pendent edge and obtain the same result.
If $G$ has at least 8 vertices and $G$ has a pendent edge $e$, the contribution
of $e$ is at least 6, and hence, $Mo(G) \ge 6$. 

Now we consider graphs of order at least $7$ with no pendent edges.

Suppose that $vw$ is a bridge of $G$, and assume that $G_v$ and
$G_w$ are the components of $G- vw$ containing $v$ and $w$, respectively.
Let $n_v \ge n_w$. We consider two cases. \\
{\bf Case 1}: $n_v=n_w$. If $\deg(v)$ (or $\deg(w)$) is at least $3$, then
the edges incident to $v$ and $w$ have contributions at
least $2$, hence $Mo(G) \ge 6$. If $\deg(v)= \deg(w)=2$, then there is a
vertex, say $z\in V(G_v)$ (or $ z\in V(G_w)$), with minimum distance from
$v$ and $\deg(z)\ge 3$.       
By applying the same argument to $z$, we get $Mo(G)>6$. That is a
contradiction. \\
{\bf Case 2}: $n_v \ge n_w +1$. Clearly we have $\phi(vw) \ge 1$. 
Let $z \in V(G_w)$ be a vertex with the minimum distance from $w$ and
$\deg(z) \ge 3$. Again, the edges incident with $z$ have contribution at least
$3$ and from this follows $Mo(G) \ge 7$, a contradiction again. 
Hence, $G$ cannot have a bridge.

It remains to show that $G$ has no cut vertices.
Suppose, to the contrary, that $v$ is a cut vertex of $G$. It is clear that, if
$\deg(v) \le 3$, then $G$ contains a bridge and hence we have $Mo(G)\ge 6$.
Thus $\deg(v) \ge 4$. Let $G -v = G_1 \cup G_2$ and let $n_i$ denote the order
of $G_i$ for $i=1,2$. We consider two cases. First we suppose $n_1=n_2$. Then
the edges incident with $v$ have contribution at least $2$. So we have
$Mo(G) \ge 8$, a contradiction. Second, let $n_1 \ge n_2 +1$. In this case,
the edges between vertex $v$ and vertices of $G_2$ have contribution at least
$3$. Thus $Mo(G) \ge 6$, again a contradiction. Hence $G$ cannot have a
cut-vertex, and must be 2-connected and 2-edge-connected.
\end{proof}

The above Lemma allows us to prove that no graph $G$ has the Mostar index 
equal to 1.

\begin{proposition} \label{Mo1}
Let $G$ be a simple connected graph. Then $Mo(G) \neq 1$.
\end{proposition}
\begin{proof}
Suppose, to the contrary, that $Mo(G)=1$ for some simple connected graph
$G$. From Lemma \ref{2con} it follows that $G$ is 2-connected. 
Since $Mo (G) \geq \varphi (e)$ for any $e \in E(G)$, there must be
exactly one edge of $e$ with nonzero contribution, and this contribution
must be equal to 1. 
Let $u$ and $v$ be the two end-vertices of $e$ with $\tr(u)=\tr(v)+1$. Since
$Mo(G)=1$, all other vertices must have the same transmission. 
From 2-connectivity of $G$ it follows that there must be two internally
disjoint paths connecting $u$ and $v$. Let
$P: v_0=u, v_1, v_2, \cdots,  v_n=v$   ($n\ge 2$) be one such path
between $u$ and $v$. Now there is a vertex $v_i$ ($0\le i \le n-1$) such
that $\tr(v_i) \neq \tr(v_{i+1})$. This implies that $|\tr(v_i)-\tr(v_{i+1})|=1$
and hence $Mo(G)\ge 2$, a contradiction.   
\end{proof}

The next unsolved case, $Mo(G) = 3$, can be realized either by a graph
having three edges with the contribution one, or one edge with contribution
1 and one edge with contribution 2. The next result eliminates the
second case.

\begin{lemma} \label{transm}
Let $G$ be a simple, 2-edge-connected graph with $Mo(G) = 3$. Then all of
its vertex transmissions must belong to a set of exactly two consecutive
positive integers.
\end{lemma}
\begin{proof}
First of all, suppose that there exists an edge $uv$ whose contribution is at
least two. Since $uv$ is not a cut edge, then there exists a $u-v$ path
$P$ not containing this edge. Then, by applying the triangle inequality on
the transmission differences across all the edges contained in $P$, we see that
the sum of their contributions is at least two. Thus, the Mostar index is at
least four, which is not possible.

These arguments assure that any graph with a Mostar index of three cannot
have an edge whose contribution is greater than one. Thus, each edge has a
contribution of either 0 or 1 and no edge can be a cut edge. Finally, since
all the edge contributions are either 0 or 1, it is easy to see that the set
of all attained transmissions must comprise consecutive integers. Suppose
that there are at least three of these integers $k - 1$, $k$, $k + 1$.
Without loss of generality, we divide the vertex set $V$ into three
nonempty subsets $V_1$, $V_2$, $V_3$ such that: \\
    $V_1$ contains the vertices whose transmission is below $k$; \\
    $V_2$ contains the vertices whose transmission is exactly $k$; \\
    $V_3$ contains the vertices whose transmission is greater than $k$. \\
Let $E_1$ be the set of edges which have one end-vertex $V_1$ and another
in $V_2$, and let $E_2$ be the set of edges that have one end-vertex in $V_2$
and the other in $V_3$. It is obvious that both of these sets are necessarily
nonempty, since the graph is connected and there can be no edge that has one
end-vertex from $V_1$ and the other from $V_3$. Moreover, it is clear that the
Mostar index cannot be lower than $|E_1| + |E_2|$. This means that at least
one of the sets $E_1$ and $E_2$ contains precisely one element. However,
this edge must be a cut edge, which is impossible, as we have shown before.
From here, it follows that the set of attained transmissions must consist
of two consecutive integers.
\end{proof}

Unlike Lemma \ref{2con}, the above Lemma does not lead to any contradictions
and cannot be used to rule out the existence of graphs with Mostar index
equal to 3. Instead, we used the obtained structural properties to narrow
the class of candidates and to do a computer-assisted search. Indeed, by
performing exhaustive search on all reasonably small ($n \leq 11$) 
2-connected and 2-edge-connected graphs, we have found one graph on
9 vertices and 2 graphs on 11 vertices with Mostar index equal to 3. As the
same search also returned ten graphs on 11 vertices with Mostar index
equal to 5, we have decided not to further pursue their characterization
along the lines of Lemma \ref{transm}. The smallest graph with the Mostar
index equal to 3 is shown in the left panel of Fig. \ref{mo3mo5}; one graph
on 11 vertices with Mostar index equal to 5 is shown in the right panel
of the same figure. Notice that the left graph is chemical; hence $p=3$
can also be realized by chemical graphs. 
\begin{figure}
\begin{center}
\includegraphics[scale=0.2]{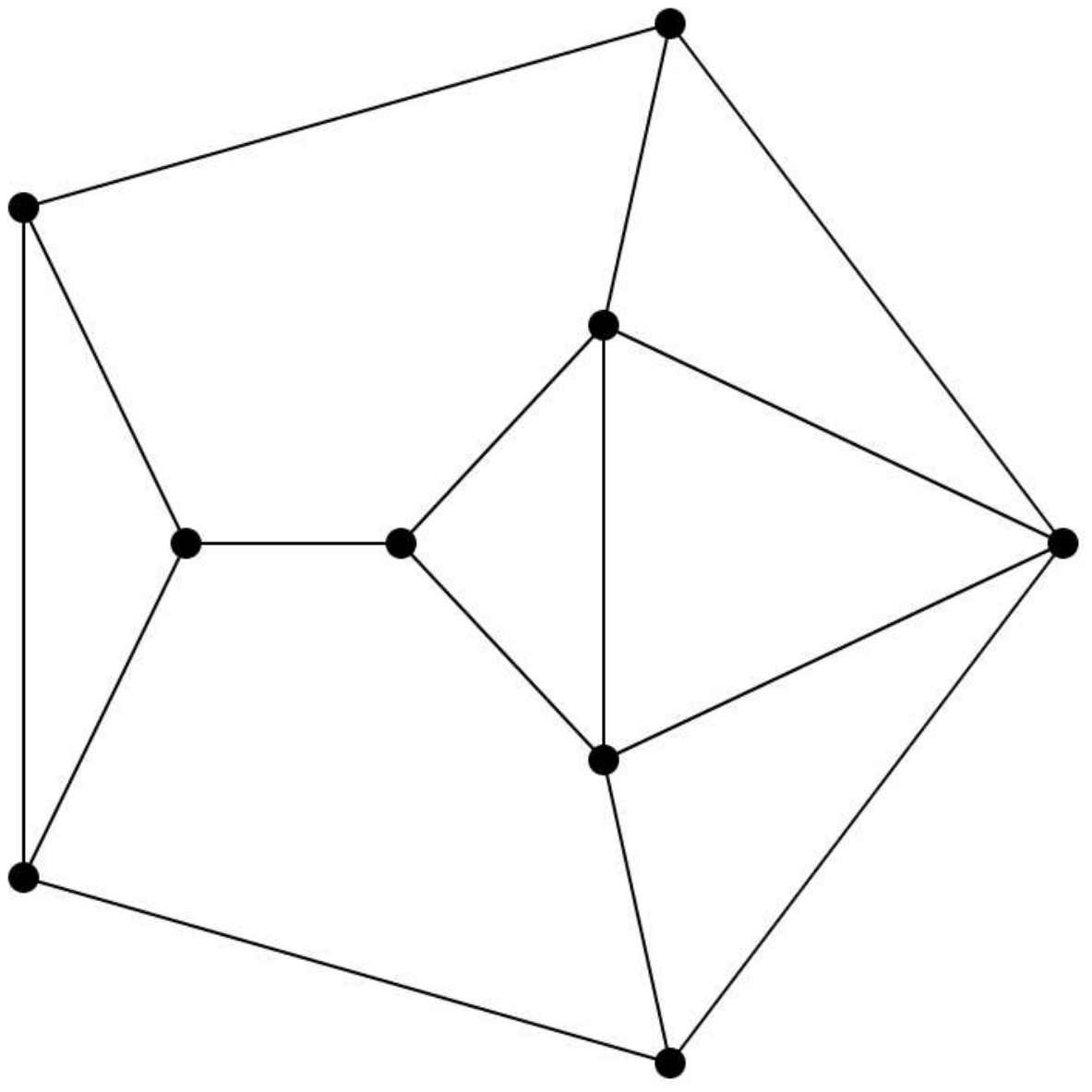}\,\,\,\quad\includegraphics[scale=0.2]{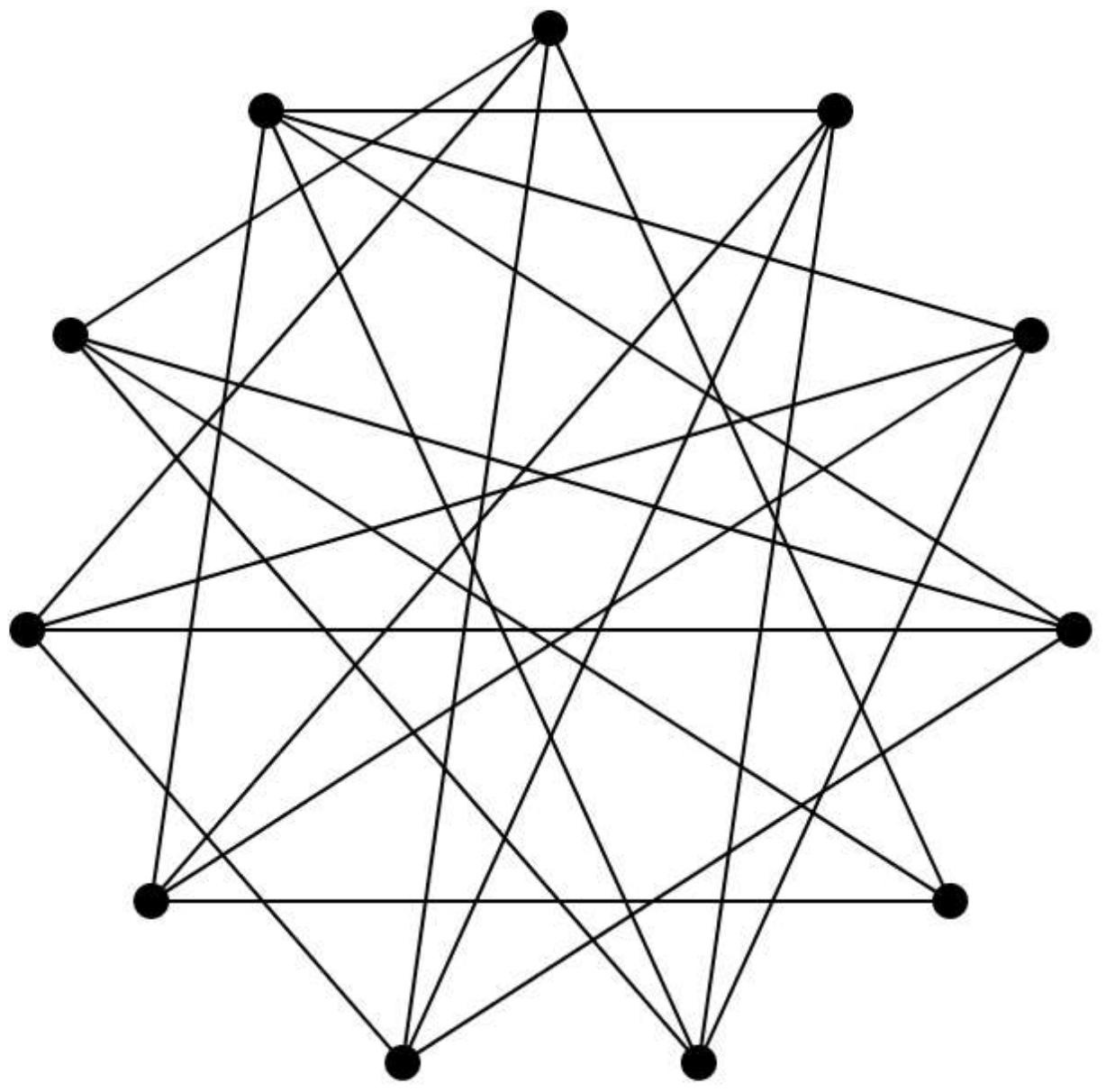}
\end{center}
\caption{Smallest graphs with $Mo(G)=3$ (left) and $Mo(G)=5$ (right).}
\label{mo3mo5}
\end{figure}

Our main result now follows by combining the search results with  the
structural characterizations provided by the above Lemmas.

We conclude this section by exhibiting another family with nice layered
structure realizing odd integers greater than one. Take the cycle $C_p$ and the
empty graph $\overline{K_p}$ on $p$ vertices. Take their join, i.e., join every
vertex of $C_p$ with every vertex of $\overline{K_p}$. Remove the edges of
one perfect matching from the obtained graph. Finally, take a copy of $K_p$
and connect its vertices with vertices of $\overline{K_p}$ by a matching of
size $p$. We leave to the readers to verify that the obtained graph on
$3p$ vertices indeed has the Mostar index equal to $p$. Examples for $p=3$ and
$p = 5$ are shown in Fig. \ref{nino35}. Notice that the graph in the left
panel is the same graph shown in the left panel of Fig. \ref{mo3mo5}.
\begin{figure}
\begin{center}
\includegraphics[scale=0.35]{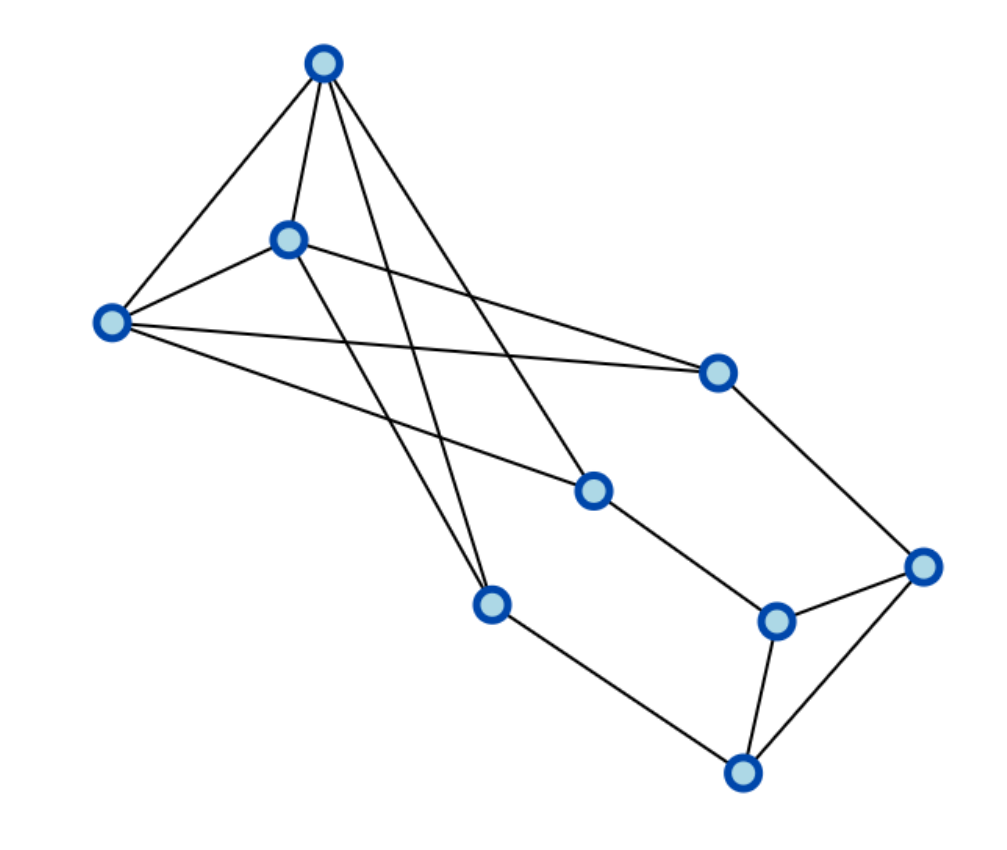}\,\,\,\quad\includegraphics[scale=0.35]{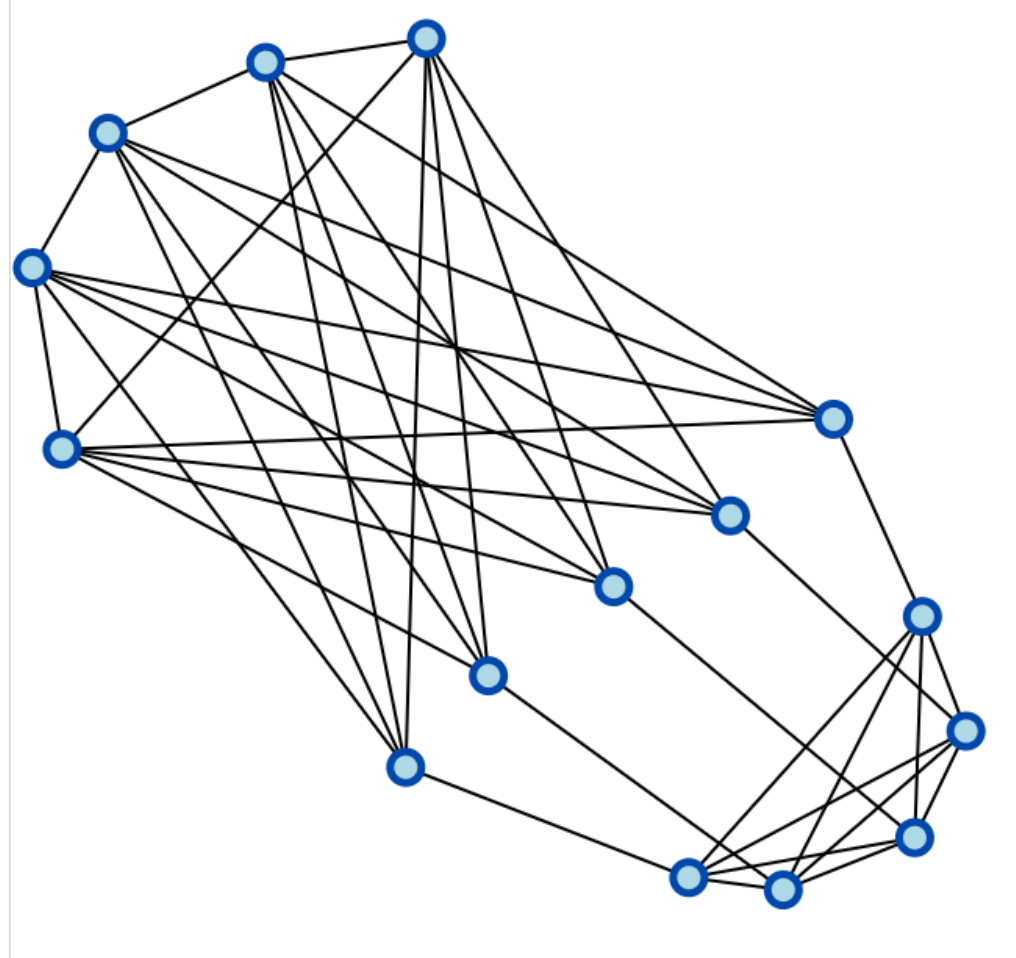}
\end{center}
\caption{Three-layer graphs with $Mo(G)=3$ (left) and $Mo(G)=5$ (right).}
\label{nino35}
\end{figure}

\section{Empirical observations}

In this section we present, in tabular and in graphic form, some results of
our exhaustive search on small graphs. Table 1 shows, in the first column,
the number of vertices $n$, and in the second column, the total number of 
connected graphs of order $n$. The next three columns show, from left to
right, the minimum value of the Mostar index, the maximum value, and the
most common value (the mode of the distribution). In all three columns, 
the number of graphs attaining the shown values is shown in parentheses.
The last column shows the average value of the Mostar index over all graphs
of order $n$.

\begin{table}[ht]
\centering
$\begin{array}[valign=c]{c|c|c|c|c|c }
n \backslash k& {\rm No.\, of\, graphs} & {\rm min\, Mo} & {\rm max\, Mo} & {\rm peak\, Mo} & {\rm average\, Mo}   \\[1.5ex]
\hline
1 &1 & 0 (1)\\
2 &1 & 0 (1)& 4 & 1\\
3 &2 & 0 (1)& 2 (1)& 2 (1) & 1. \\
4 &6 & 0 (2)& 6 (1)& 4 (3)& 3. \\
5 &21 & 0 (2)& 12 (2)& 8 (6)& 6.857\\
6 &112 & 0 (5)& 24 (1)& 12 (21)& 11.67\\
7 &853 & 0 (4)& 40 (1)& 20 (95)& 18.129\\
8 &11117 & 0 (15)& 60 (2)& 24 (847)&  25.402\\
9 &261080 & 0 (23)& 90 (1)& 32 (14652)& 33.741\\
10 &11716571 & 0 (120)& 126 (1)& 40 (545116)& 43.174\\
11 &1006700565 & 0 (313)& 168 (2)& 52 (39598746)& 53.96\\
\end{array}$
\caption{Some statistics of Mostar indices of small graphs.}\label{tab:r_k(n)}
\end{table}

The sequence in the second column is the number of simple connected graphs
on $n$ vertices; it appears as A001349 in the {\em On-Line Encyclopedia of 
Integer sequences} \cite{oeis}. All elements in the third column outside
of parentheses are zero,
reflecting the fact that distance-balanced graphs exist for all positive
numbers of vertices. More interesting in this column is the sequence in 
parentheses, enumerating such graphs. It does not (yet) appear in the OEIS.
In the fourth column, both sequences seem to be interesting. The one in the
parentheses, periodic with period 3, indicates that the divisibility by 3
plays a role in the number (and in the structure) of graphs maximizing the
Mostar index. This is confirmed by the fact that the other sequence appears
as A307559 in the OEIS; its $n$-th term is given there as
$$\left \lfloor \frac{n}{3} \right \rfloor \left ( n - \left \lfloor \frac{n}{3} \right \rfloor \right ) \left ( n - \left \lfloor \frac{n}{3} \right \rfloor -1 \right ).$$
This expression is exactly the Mostar index of the split graphs $S_{2n/3,n/3}$
conjectured
in \cite{Tom2018} to maximize it. (A {\em split graph} $S_{a,b}$ is obtained
by joining every vertex of the complete graph $K_a$ to every vertex of an
empty graph on $b$ vertices.) It was shown in \cite{geneson} that the 
conjecture does not hold and that greater values are achieved for complete
bipartite graphs $K_{\lfloor \alpha /n \rfloor, n - \lfloor \alpha /n \rfloor}$
for $\alpha = \frac{1}{2} \left ( 1 - \frac{1}{\sqrt 3} \right )$ for big
enough $n$, but the conjectured values are still true in the considered range.
We notice that that neither of sequences in the peak Mo column
appears in the OEIS; at the moment we have no plausible explanation for 
the numbers appearing there. 

Table 1 contains just the extremes and the mode of distribution for a given
number of vertices. In Fig. \ref{fig-distr} we plot the full distribution
of Mostar indices for connected graphs on 8 vertices. The most striking
property is the strong difference between the number of graphs with
\begin{figure}
\begin{center}
\includegraphics[scale=0.4]{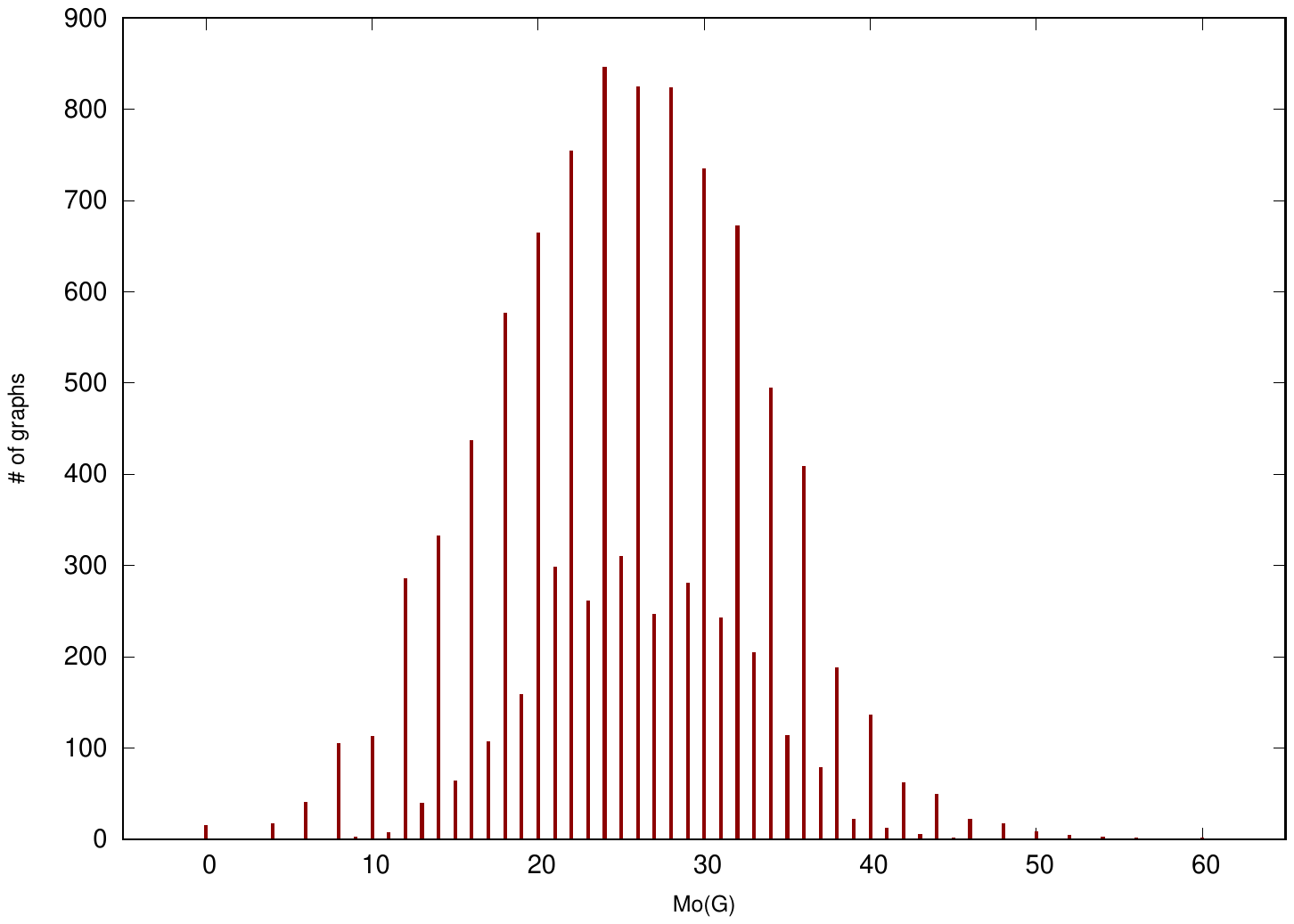}
\end{center}
\caption{Distribution of values of the Mostar index for graphs with 8 vertices.}
\label{fig-distr}
\end{figure}
even and with odd values of the Mostar index. Both sub-distributions have
the same bell-like shape, with peaks somewhat below the mid-range, but the
odd sub-distribution has a much lower peak (311 vs. 847) and significantly
narrower range (9 -- 47 vs. 0 -- 60).
Another interesting feature is the non-unimodality, i.e., the existence of
multiple local maxima, more strongly pronounced for the odd values.
For other (large enough) numbers of vertices, distributions have the same
general shape and exhibit the same properties with respect to the parity and
multi-modality.

Some other interesting properties of the Mostar index distributions are
revealed in our Table 2. There we tabulated possible values of Mostar
indices, starting from 2, and in each row, for a given $Mo (G)$, in
column $n$ we give the number of graphs on $n$ vertices attaining the 
given value. We have limited the table to the values $ 2 \leq Mo (G) \leq
10$ because they capture all properties observed for the larger values.
The row 1 is omitted since 1 is not realizable, and the row 0 is omitted
since it would contain exactly the numbers in parentheses in the third
column of Table 1. The columns for $n = 1,2$ are also omitted since the
only possible value of Mostar index of graphs on 1 and on 2 vertices is zero.
The missing entries in the shown columns are indicated by --.
\begin{table}[ht]
\centering
$\begin{array}[valign=c]{c|c|c|c|c|c|c|c|c|c }
Mo (G) \backslash n& 3 & 4 & 5 & 6 & 7 & 8 & 9 & 10 & 11   \\[1.5ex]
\hline
2 & 1 & - & 1 & 1 & 2 & - & - & - & 1 \\
3 & - & -& -& -& - & - & 1 & -& 2 \\
4 & - & 3 & 2 & 3 & 12 & 18 & 56 & 34 & 3278 \\
5 & - & - & - & - & - & - & - & - & 10 \\
6 & - & 1 & 3 & 3 & 14 & 41 & 103 & 618 & 4483 \\
7 & - & - & 2 & - & 1 & - & - & - &  26 \\
8 & - & - & 6 & 16 & 31 & 105 & 387 & 2132 & 14348 \\
9 & - & - & - & 1 & 3 & 5 & 15 & 19 & 459 \\
10 & - & - & 3 & 8 & 24 & 113 & 480 & 4715 & 31259
\end{array}$
\caption{Another view of distributions of Mostar indices for small graphs.}\label{tab:r_k(n)}
\end{table}

There are two interesting observations to be noted in Table 2. The first
is that the sequence starting with 3,9,4,11,4,5,5,6,5 does not appear in the
OEIS. This sequence shows the smallest number of vertices of a graph
attaining a given value of $Mo (G)$. The second is that, for any given number
of vertices, there seem to be many more graphs with even than with odd 
values of the Mostar index. It would be interesting to have a plausible
explanation of those facts.

We close this section by listing some other observations resulting from
our search on small graphs. As of now we do not have proofs for any of the 
following claims. Counterexamples to any, or at all, of them are welcome.
\begin{itemize}
\item The Mostar index of regular graphs is even;
\item The graphs with $Mo (G) = 3$ are bi-degreed, with max degree 4 and
min degree 3;
\item The graphs with $Mo (G) = 5$ have the min degree 3 and the max degree 5.
If the max degree is 5, then there is exactly one vertex of degree 5 and
exactly one vertex of degree 3;
\item Each graph with Mostar degree 3 or 5 contains a triangle;
\item Only one graph with $Mo (G) = 3$ is planar.
\end{itemize}

\section{Infinite realizability}

A nonnegative integer $p$ is {\em infinitely realizable} by a graph-theoretical
invariant $I$ if there are infinitely many graphs $G$ such that $I(G) = p$.
(When the invariant is clear from the context, we just speak about infinite
realizability of $p$.)
Since all cycles and complete graphs have Mostar index equal to 0, it is
clear that 0 is infinitely realizable by the Mostar index. We have shown
that 1 is not realizable. What about the integers greater than one? The next
result shows that all even nonnegative integers are infinitely realizable.
\begin{proposition}
For any even nonnegative integer $p$ there are infinitely many graphs $G$ 
such that $Mo (G) = p$.
\end{proposition}
\begin{proof}
Let $p = 2m$ be an even nonnegative integer. We consider the graph $G(p,k)$
constructed from $2k+1$ copies of complete graphs $K_{m+1}$ and $2k+3$
copies of $K_1$, connected in a way shown in Fig. \ref{eveninf}. The claim
will follow if we show that $Mo (G) = p$ for any choice of $k \geq 1$. So,
label the building blocks of $G(p,k)$ cyclically in the way shown in Fig.
\ref{eveninf}. We say that all vertices belonging to the graph labeled by
$j$ are at level $j$.
Our graph $G(p,k)$ has $4k+4$ levels in total, out of which there are $2k+1$
levels with $K_{m+1}$ and $2k+3$ levels consisting of $K_1$. Notice that 
any $K_1$ is connected with all vertices of complete graphs at neighboring
levels. Now, if we divide all $4k+4$ levels into two consecutive halves of
$2k+2$ each, it is clear that each half can have at most $k+1$ levels with
$K_{m+1}$, since there can be no two consecutive such levels. However, the
total number of such levels is $2k+1$, which is only possible if one half
contains $k$ levels consisting of $K_{m+1}$ and $k+2$ levels of $K_1$,
while the other contains $k+1$ levels of $K_{m+1}$ and $k+1$ levels of $K_1$.
\begin{figure}
\begin{center}
\includegraphics[scale=0.35]{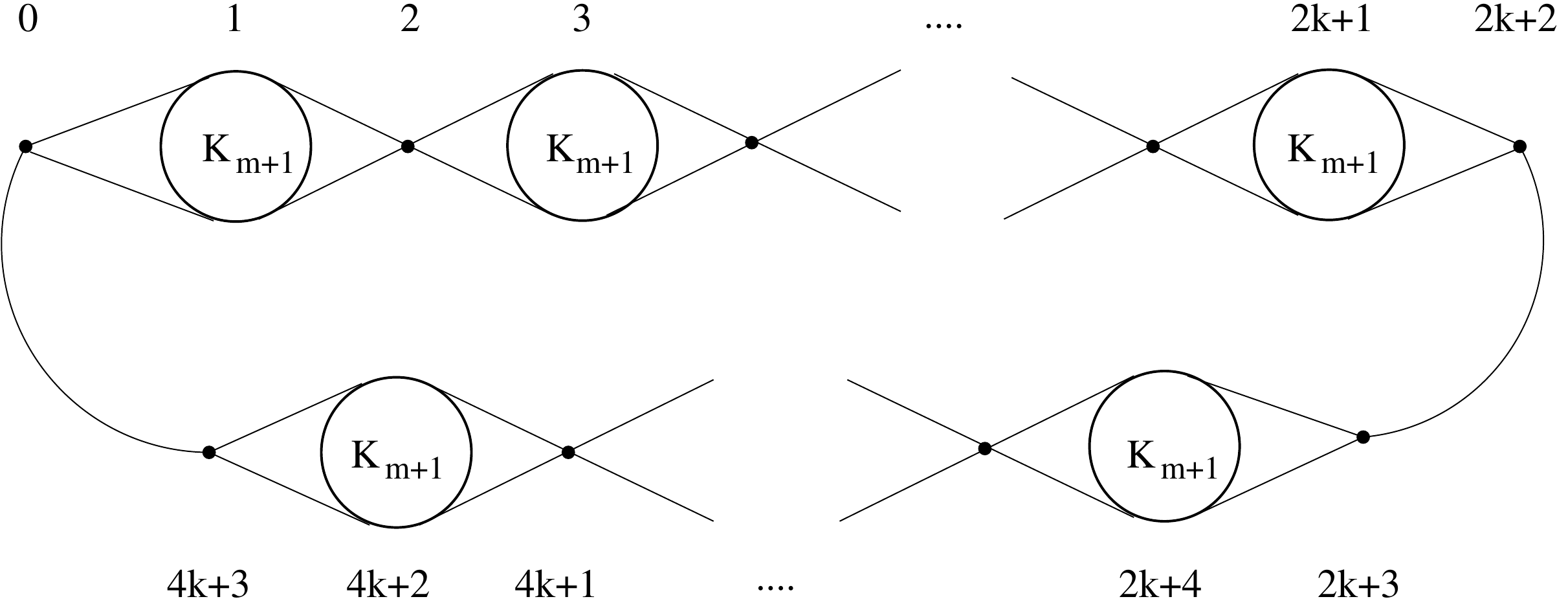}
\end{center}
\caption{Infinite realizability of even integers.}
\label{eveninf}
\end{figure}

Let $x$ be a vertex from level $j$ and let $y$ be a vertex from level $j+1$,
so that level $j$ has 1 vertex, while level $j+1$ has $m+1$ vertices.
Bearing in mind the $j, j-1, j-2, ..., j-(2k-1)$ levels, it is not difficult
to see that there are $k$ levels of $K_{m+1}$ among them. If there were
$k+1$ such levels, this would be possible only if their numbers of vertices
were $1, m+1, 1, m+1, 1, m+1, ..., 1, m+1$, respectively, which is impossible,
since level $j+1$ is also a $K_{m+1}$-level. Hence, this would yield 
$2k+3$ alternating consecutive levels starting and ending with a
$K_{m+1}$-level, which does not exist. Thus, the number of vertices closer
to $x$ than to $y$ equals
$$k \cdot (m+1) + (k+2) \cdot 1 = mk + 2k + 2.$$
On the other hand, the levels $j+1, j+2, j+3, ..., j+(2k+2)$ must consist of
$k+1$ $K_{m+1}$-levels and $k+1$ $K_1$-levels. All of these vertices are
closer to $y$ than to $x$, except for the other ones from the level $j+1$
which are not $y$. Thus, the total number of vertices closer to $y$ than to
$x$ equals
$$(k+1) \cdot (m+1) + (k+1) \cdot 1 - (m+1- 1) = mk + 2k + 2.$$
Thus, $\tr(x) = \tr(y)$, and contributions of all such edges are zero.
It can be analogously shown that the same
conclusion can be made if level $j$ has $m+1$ vertices and level $j+1$ has
1 vertex. So, the only candidates for nonzero contributions are the edge 
between the levels $2k+2$ and $2k+3$, and the edge between the levels
$4k+3$ and $0$. It can be easily shown that both have the contribution of
$m$. Hence the sum of all contributions, $2m = p$, depends only on $m$ and not
on $k$, enabling us to construct an infinite class of graphs with a given
even value of the Mostar index.
\end{proof}
We notice that the above construction works also for the case $p=0$. In that
case, our $G(0,k)$ becomes $C_{4k+4}$, with Mostar index equal to zero for
any $k$.

At the present, we do not have an analogous result for odd integers. Moreover,
we do not know of a single odd nonnegative integer which is infinitely 
realizable. Hence, it makes sense to pose the following questions.

{\bf Problem 1} Are there any odd integers greater than one infinitely
realizable by the Mostar index? Are all of them infinitely realizable? 
If not all, which ones are?

From Table 2 one cannot infer too much about those problems. The row 9, as 
well as the rows 11, 13, and 15 (not shown in Table 2) all exhibit increasing
behavior similar to the even rows, offering a reason to expect that all odd
integers are also infinitely realizable.

\section{Consequences and further developments}

In this section we present some partial results with respect to several
restricted versions of the inverse problem. 

\subsection{Minimal realizability}

Since all integers greater than one are realizable by the Mostar index, the
following problem is well defined.

{\bf Problem 2}  For a given integer $p > 1$, what is the minimum number
$n(p)$ of vertices of a graph $G$ with $Mo(G) = p$? Among all graphs on
$n(p)$ vertices which realize $p$, what is the smallest number of edges
$m(p)$? 

From Lemma \ref{sviosim} it follows that every integer $p \geq 6$ is 
realizable by a graph on $\left \lceil \frac{p+3}{2} \right \rceil$ vertices.
It means that at most $\left \lfloor \frac{p+3}{2} \right \rfloor$ initial
terms in the $p$-th row of Table 2 can be zero. For even integers of the
form $p = (n-1)(n-2)$ the first nonzero entry must appear much earlier,
around $\sqrt{p}$, since $p$ is realizable by the $n$-vertex star $S_n$.
For the integers realizable by split graphs $S_{2n/3,n/3}$, the first
nonzero entry appears around $3 \sqrt[3]{p}$. The exact form of $n(p)$
is still an open problem, pending the solution of the maximum possible
value of $Mo (G)$ for graphs on a given number of vertices. See
\cite{geneson,mikl,mikl1} for some results in that direction.

For finitely realizable integers, if they exist, it makes sense to ask
for the largest graphs realizing them, both in terms of the number of
vertices and of the number of edges.

{\bf Problem 3} If an odd integer $p$ is finitely realizable, what is the 
largest number of vertices $N(p)$ of a graph $G$ realizing $p$? What is the
largest number of edges $M(p)$?

\subsection{Restricted realizability}

By restricted realizability we mean realizability by a particular class
of graphs. For example, which integers greater than one can be realized
by trees, or by unicyclic graphs, or by any other class? Here we 
present the complete answer for trees. 

\begin{proposition} \label{realtrees}
All even nonnegative integers are realizable by trees.
\end{proposition}
\begin{proof}
We will prove a stronger result, that all even nonnegative integers are
actually, realizable by chemical trees, i.e., by trees having the maximum
degree 4. We start by noticing that $p=0$ is realized by $K_2 = P_2$, which
is a chemical tree. Further, since 
$Mo(P_3)=2$, $Mo(P_4)=4$ and $Mo(S_4)=6$, it suffices to consider $p\geq 8$.
We can routinely check that
$\Big\lfloor\frac{(n-1)^2}{2}\Big\rfloor+2\Big\lfloor\frac{n-2}{2}\Big\rfloor+2=\Big\lfloor\frac{n^2}{2}\Big\rfloor$
holds for any odd or even integer $n\geq 5$. Therefore, there are exactly
$\Big\lfloor\frac{n-2}{2}\Big\rfloor+1$ even numbers in the interval
$\Big[\Big\lfloor\frac{(n-1)^2}{2}\Big\rfloor,\Big\lfloor\frac{n^2}{2}\Big\rfloor\Big)$.
By Theorem \ref{small}, we have
$Mo(P_n)=\Big\lfloor\frac{(n-1)^2}{2}\Big\rfloor$ and
$Mo(T_n(1,k,n-2-k))=\Big\lfloor\frac{(n-1)^2}{2}\Big\rfloor+2k$
with $1\le k\le \lfloor\frac{n-2}{2}\rfloor$. Note that any positive integer
$p$ lies in an interval
$\Big[\Big\lfloor\frac{(n-1)^2}{2}\Big\rfloor,\Big\lfloor\frac{n^2}{2}\Big\rfloor\Big)$
for some integer $n$. The result follows immediately.
\end{proof}
By combining the above results with Lemma \ref{even}, we obtain the 
complete solution for realizability by trees.
\begin{theorem}
A nonnegative integer $p$ is the Mostar index of a tree if and only if $p$ is
even.\end{theorem}
We are not aware of any studies of realizability of integers as Mostar
indices of other classes of graphs. The most obvious candidates are the uni-
and the bicyclic graphs, and also the other graphs with low cyclomatic number.
We believe that the cactus graphs are also promising.

{\bf Problem 4} Which nonnegative integers can appear as Mostar indices of
graphs with the cyclomatic number equal to $c$ for a given positive integer
$c$?

{\bf Problem 5} Which nonnegative integers are Mostar indices of proper
cacti, i.e., of graphs whose each block is a cycle?

{\bf Problem 6} Which nonnegative integers are Mostar indices of bipartite
graphs?

\subsection{Miscellaneous open problems}

The list of problems in the previous subsections is far from exhaustive.
We have settled some open questions, but many more remain unanswered. Here
are just two of them.

{\bf Problem 7} Why are graphs with even Mostar indices far more common
than the graphs with odd values?

{\bf Problem 8} What can be said about the position of the maximum
of the distribution of $Mo(G)$ for a given number of vertices? In other
words, which values of the Mostar index is the most common among graphs on a
given number of vertices? Which values within the range cannot appear?

Our empirical results suggest that any graph realizing 5 by its Mostar
index has the maximum degree 5, but we have no proof of this fact. It is,
in principle, possible that 5 could be also realized by a chemical graph,
although we deem it unlikely. Hence our last problem.

{\bf Problem 9} Is $p=5$ the only nonnegative integer greater than one
not realizable by chemical graphs?

We close the paper by noting that the inverse problems for the weighted
versions of the Mostar index, for the total Mostar index, and for other
generalizations mentioned in \cite{alidoslic}, are still widely open. 

\section*{Acknowledgments}
I. Damnjanovi\'c is supported by Diffine LLC.
T. Do\v{s}li\'c gratefully acknowledges partial support by Slovenian ARRS
via grant no. J1-3002 and by COST action CA21126 NanoSpace.
This work is supported in part by the Slovenian
Research Agency (research program P1-0294 and research projects N1-0140,
J1-2481). K. Xu is supported by NNSF of China (Grant No. 12271251).


\begin{thebibliography}{99}

\bibitem{alidoslic}
A. Ali, T. Do\v{s}li\'c,
\newblock Mostar index: Results and perspectives,
\newblock {\em Appl. Math. Comput.} 404 (2021) 126245.

\bibitem{bon-1976}
J. A. Bondy, U. S. R. Murty,
\newblock {\em Graph Theory with Applications} (Macmillan Press, New York, 1976)

\bibitem{dehgardi}
N. Dehgardi, M. Azari, 
\newblock More on Mostar index, 
\newblock {\em Appl. Math. E-Notes} 20 (2020) 316--322.

\bibitem{Tom2018}
T. Do\v{s}li\'{c}, I. Martinjak, R. \v{S}krekovski, S. Tipuri\'c Spu\v{z}evi\'{c}, I. Zubac,
\newblock Mostar index,
\newblock {\em J. Math. Chem.} 56 (2018) 2995--3013.

\bibitem{gxd}
F. Gao, K. Xu, T. Do\v{s}li\'c,
\newblock On the difference of Mostar index and irregularity of graphs,
\newblock {\em Bull. Malays. Math. Sci. Soc.} 44 (2021) 905--926.

\bibitem{geneson}
J. Geneson, S.-F. Tsai,
Peripherality in networks: theory and applications,
\newblock {\em J. Math. Chem.} 60 (2022) 1021--1079.

\bibitem{ghalav}
A. Ghalavand, A. Ashrafi, M. Hakimi-Nezhad,
\newblock On Mostar and edge Mostar indices of graphs,
\newblock {\em J. Math.} 2021 (2021) 6651220.

\bibitem{Gu-1986}
\newblock I. Gutman, O. E. Polansky,
\newblock {\em Mathematical Concepts in Organic Chemistry}, (Springer, Berlin, 1986)

\bibitem{guttrina}
I. Gutman, N. Trinajsti\'c,
\newblock Graph theory and molecular orbitals. Total $\pi$-electron energyof alternant hydrocarbons,
\newblock {\em Chem. Phys. Lett.} 17 (1972) 535--538.

\bibitem{je-2008}
J. Jerebic, S. Klav\v{z}ar, D. F. Rall,
\newblock Distance-Balanced Graphs,
\newblock {\em Ann. Comb.} 12 (2008) 71--79.

\bibitem{mikl}
\v{S}. Miklavi\v{c}, J. Pardey, D. Rautenbach, F. Werner,
\newblock Maximizing the Mostar index for bipartite graphs and split graphs, 
\newblock {\em Discrete Optim.} 48 (2023) 100768.

\bibitem{mikl1}
\v{S}. Miklavi\v{c}, J. Pardey, D. Rautenbach, F. Werner,
\newblock Bounding the Mostar index,
\newblock {\em Discrete Math.} 347 (2024) 113704.

\bibitem{oeis}
OEIS Foundation Inc. (2023),
\newblock The On-Line Encyclopedia of Integer Sequences,
\newblock Published electronically at https://oeis.org.

\bibitem{randic}
M. Randi\'c,
\newblock Characterization of molecular branching,
\newblock {\em J. Am. Chem. Soc.} 97 (1975) 6609--6615.

\bibitem{te-2019}
\newblock A. Tepeh,
Extremal bicyclic graphs with respect to Mostar index,
\newblock {\em Appl. Math. Comput.} 355 (2019) 319--324.

\bibitem{tra-2019}
N. Tratnik,
\newblock Computing the Mostar index in networks with applications to molecular graphs,
\newblock {\em Iranian J. Math. Chem.} 12 (2021) 1--18.

\bibitem{wang-2006}
H. Wang, G. Yu,
\newblock All but 49 numbers are Wiener indices of trees,
\newblock {\em Acta Appl. Math.} 92 (2006) 15--20.

\bibitem{wie-1947}
H. Wiener, 
\newblock Structural determination of paraffin boiling points,
\newblock {\em J. Am. Chem. Soc.} 69 (1947) 17--20.

\bibitem{xu-2019}
K. Xu, F. Gao, K. C. Das, N. Trinajsti\'{c},
\newblock A formula with its applications on the difference of Zagreb indices of graphs, 
\newblock {\em J. Math. Chem.} 57 (2019) 1618--1626.

\bibitem{yu-2019}
A. Yurtas, M. Togan, V. Lokesha, I.N. Cangul, I. Gutman,
\newblock Inverse problem for Zagreb indices,
\newblock {\em J. Math. Chem.} 57 (2019) 609--615.

\end{thebibliography}
\end{document}